\documentclass[11pt,a4paper]{amsart}

\usepackage{amsmath,amssymb,bbm}

\usepackage{graphicx,subcaption,microtype}
\usepackage{a4wide}

\usepackage[margin=2.5cm]{geometry}

\usepackage{tikz}
\usetikzlibrary{arrows,backgrounds,decorations,automata}
\usetikzlibrary{lindenmayersystems}
\pgfdeclarelindenmayersystem{cantor set}{
  \rule{F -> FfF}
  \rule{f -> fff}
}

\theoremstyle{plain}
\newtheorem{theorem}{Theorem}[section]

\theoremstyle{definition}
\newtheorem{remark}[theorem]{Remark}

\newcommand{\ee}{\mathrm{e}}
\newcommand{\ii}{\mathrm{i}}

\newcommand{\dd}{\,\mathrm{d}}

\begin{document}

\title[A sequential view of self--similar measures]
{A sequential view of self--similar measures,\\
\vskip .1cm
\emph{or}, What the ghosts of Mahler and Cantor\\
\vskip .1cm
can teach us about dimension} 

\author[M. Coons]{Michael Coons}
\author[J. Evans]{James Evans}

\address{School of Mathematical and Phys.~Sciences, University of
  Newcastle, \newline \hspace*{\parindent}University Drive, Callaghan
  NSW 2308, Australia}
\email{michael.coons@newcastle.edu.au,james.evans10@uon.edu.au}


\thanks{We acknowledge the support of the Commonwealth of Australia.}

\maketitle

\vspace{.5cm}
\begin{abstract} We show that missing $q$-ary digit sets $F\subseteq[0,1]$ have corresponding naturally associated countable binary $q$-automatic sequence $f$. Using this correspondence, we show that the Hausdorff dimension of $F$ is equal to the base-$q$ logarithm of the Mahler eigenvalue of $f$. In addition, we demonstrate that the standard mass distribution $\nu_F$ supported on $F$ is equal to the ghost measure $\mu_f$ of $f$. 
\end{abstract}
\vspace{.8cm}

\section{Introduction}

This story starts with an uncountable set and a countable sequence, their respective dimension and asymptotical behaviour, and describes the setting in which these objects and properties---and their generalisations---coalesce.

The uncountable set is the ubiquitous standard middle-thirds Cantor set, $C$, the self-affine set that is the unique attractor of the iterated function system $\mathcal{S}_C:=\{S_1,S_2\}$, where the functions $S_1,S_2:[0,1]\to[0,1]$ are the affine contractions $$S_1(x)=\frac{x}{3}\qquad\mbox{and}\qquad S_2(x)=\frac{2+x}{3}.$$ The standard construction of $C$ is via the production of ``level sets'' $E_k$, for $k\geqslant 0$, where $E_k:=S_C^k([0,1])$, the function $S_C$ being defined by $S_C(E)=S_1(E)\cup S_2(E)$ for any set $E$, and $S_C^k$ denoting $k$-fold composition of $S_C$ with itself; see Figure \ref{cantf}. Each level set $E_k$ has Lebesgue measure $\lambda(E_k)=(2/3)^k$, so that $C$ has $\lambda$-measure zero. Nonetheless, $C$ has some size; it has Hausdorff dimension $\log_3(2)$, where $\log_3(\cdot)$ denotes the ternary logarithm. Further, one can construct a natural mass distribution $\nu_C$, called the Cantor measure, supported on $C$ by repeated subdivision, at each level spreading the mass equally among the subintervals of $E_k$ and taking $k$ to infinity. In this way, the $\lambda$-singular continuous measure $\nu_C$ is the (weak) limit of $\lambda$-absolutely continuous measures. We thus arrive to the paradigmatic example of a fractal set $C$ and a related mass distribution $\nu_C$. Much of the theories of fractal geometry and discrete dynamical systems have been developed pursuing generalisations of this example; see Falconer's definitive monograph {\em Fractal Geometry} \cite{Falconer} for details and discussions.

\newpage
\begin{figure}[ht]
\begin{center}
\begin{tikzpicture}
\foreach \order in {0,...,5,6}
\draw[yshift=-\order*17pt]  l-system[l-system={cantor set, axiom=F, order=\order, step=300pt/(3^\order)}];
\put(0,2){{\tiny $0$}}
\put(98,4){{\tiny $\frac{1}{3}$}}
\put(198,4){{\tiny $\frac{2}{3}$}}
\put(298,2){{\tiny $1$}}
\put(305,-3){{$E_0$}}
\put(305,-20){{$E_1$}}
\put(305,-37){{$E_2$}}
\put(305,-54){{$E_3$}}
\put(305,-71){{$E_4$}}
\put(307,-90){{$\vdots$}}
\put(305,-105){{$C$}}
\draw[white,line width=10pt] (0,-3) -- (10.55,-3);
\end{tikzpicture}
\end{center}
\vspace{-.1cm}
\caption{Iterated construction of the Cantor set $C$, where $E_k=S_C^k([0,1]).$}
\label{cantf}
\end{figure}

${}$
\vspace{-.8cm}

Transitioning from the uncountable to the countable, we consider the $3$-automatic binary sequence that is the infinite fixed point $\varrho_c^\infty(1)$ of the substitution $\varrho_{c}$ defined by
\begin{equation*}
   \varrho_{c}: \, \begin{cases} 1 \mapsto 101 \\ 
     0\mapsto 000\, . \end{cases}
\end{equation*} We define the {\em Cantor sequence} $\{c(n)\}_{n\geqslant 0}$ as the $n$th entry of $\varrho_c^\infty(1)$. The construction of  $\varrho_c^\infty(1)$ is reminiscent of the construction of the Cantor set $C$, the iterate $\varrho_c^k(1)$ playing the part of the level set $E_k$; see Figure \ref{fig:seq}. It is easy to see that $c(n)=1$ precisely when the ternary expansion of the integer $n$ contains no ones, so that analogy to the Cantor set is clear. See Allouche and Shallit \cite{ASbook} for details on automatic sequences.
\begin{figure}[h]
\vspace{-.3cm}
\begin{align*} 
\varrho_{\rm c}^0(1)&= 1\\
\varrho_{\rm c}^1(1)&= 101\\
\varrho_{\rm c}^2(1)&= 101000101\\
\varrho_{\rm c}^3(1)&= 101000101000000000101000101\\
&\vdots\\
\varrho_{\rm c}^\infty(1)&=  
101000101000000000101000101000000000000000000000000000101000\cdots
\end{align*}
\vspace{-.8cm}
\caption{Iterated construction of the Cantor sequence $c$}
\label{fig:seq}
\end{figure}
\vspace{-.2cm}

Using the definition, since $c$ is $3$-automatic, its generating function is a Mahler function, and one obtains immediately that the generating function of the Cantor sequence can be written as an infinite product, $$M_c(z):=\sum_{n\geqslant 0}c(n)z^n=\prod_{j\geqslant 0}\left(1+z^{2\cdot 3^j}\right).$$ A result of Bell and Coons \cite{BC2017} provides, as $z\to 1^-$, the asymptotics $$M_c(z)\asymp\frac{1}{(1-z)^{\log_3(2)}}(1+o(1)).$$ Here, the number $2$ in the quantity $\log_3(2)$ is the Mahler eigenvalue $\lambda_c$ of $M_c(z)$ (details and definitions are provided in the next section) obtained from the Mahler-type functional equation satisfied by $M_c(z)$, $$M_c(z)-(1+z^2)M_c(z^3)=0.$$ Note that the divergent behaviour of $M_c(z)$ is governed by the Hausdorff dimension of $C$; that is, $\dim_H C=\log_3(2)$. This relationship, generalised, is our first result.

\begin{theorem}\label{main1} Let $q\geqslant 2$ be an integer, $A\subseteq\{0,\ldots,q-1\}$ containing $0$ and $F$ be the subset of $[0,1]$ consisting of numbers that can be represented in base-$q$ using only $q$-ary digits form $A$. Then, to $F$ there is a naturally associated $q$-automatic binary sequence $f$ such that $$\dim_{\rm H}F=\log_q(|A|)=\log_q(\lambda_{f}),$$ where $\lambda_{f}$ is the Mahler eigenvalue of the generating function $M_f(z)$ of $f$ and $\log_q(\cdot)$ denotes the base-$q$ logarithm.
\end{theorem} 

Returning to our motivating example, analogous to the Cantor measure $\nu_C$, we may construct a mass distribution $\mu_c$ associated to the sequence $c$, called the ghost\footnote{The term {\em ghost measure} was introduced by the second author \cite{EvansPRE}, inspired by the phrase ``the ghost of departed quantities'' as appearing in Berkeley's critique of calculus. Evans \cite{EvansPRE} writes, ``The terms of the sequence are (usually) much smaller than the sum of the terms, so the individual pure points of the $\mu_n$ are driven to zero by the averaging as $n$ tends to infinity. The measure $\mu$ is the ghost of the departed pure points of the $\mu_n$.''} measure. This construction was introduced by Baake and Coons \cite{BCaa}. Here, for each $n$, we take the sequence $c$ up to $3^n-1$ (a kind of `fundamental region') and reinterpret its (renormalised) values as the weights of a pure point probability measure $\mu_{c,n}$ on the torus $\mathbb{T}=[0,1)$ with support ${\rm supp}(\mu_{c,n})=\left\{\frac{j}{3^n}:0\leqslant j<3^n,\, c(j)=1\right\}.$ That is, $$\mu_{c,n}:=2^{-n}\sum_{j=0}^{3^n-1}c(j)\, \delta_{j/3^n},$$ where $\delta_x$ denotes the unit Dirac measure at $x$. The ghost measure $\mu_c$ is the (weak) limit of the sequence $(\mu_{c,n})_{n\geqslant 0}$ as $n\to\infty$. With the obvious generalisation in notation, we state our second result.

\begin{theorem}\label{main2} With the notation above and assumptions of Theorem \ref{main1}, the standard mass distribution $\nu_F$ supported on $F$ is equal to the ghost measure $\mu_f$.
\end{theorem}

\noindent In particular, the Cantor measure $\nu_C$ is equal to the ghost measure $\mu_c$. In this way, Theorem~\ref{main2} provides an alternative construction of standard mass distributions supported on fractals, and does so in a way that takes the countable to the uncountable.

\begin{remark} Self-similar sets---attractors of affine contractions---have been used by many authors to model physical phenomena; see Mandelbrot \cite{M1977} for a detailed history and references. The study of self-similar sets was put into a rigorous general framework by Hutchinson \cite{H1981} and the study of self-similar measures and their Fourier transforms was studied in detail by Strichartz \cite{S1990,S1993}; see also Makarov \cite{M1993,M1994}. See Baake and Moody \cite{BM2000} for the generalisation to compact families of contractions.
\end{remark}

\section{Hausdorff dimension and the Mahler eigenvalue}

In this section, we prove Theorem \ref{main1}. 

To this end, Let $q\geqslant 2$ be an integer, $A\subseteq\{0,\ldots,q-1\}$ containing $0$ and $F$ be the subset of $[0,1]$ consisting of numbers that can be represented in base-$q$ using only $q$-ary digits form $A$. We enumerate the set $A$ as $$0=a_1<a_2<\cdots<a_{m-1}<a_m\leqslant q-1.$$ Then $F$ is the unique attractor of the iterated function system $\mathcal{S}=\{S_1,\ldots ,S_m\}$, where the functions $S_i$ are given by \begin{equation}\label{S}S_i(x)=\frac{x+a_i}{q}.\end{equation} Moreover, we have \cite[Thm.~9.3]{Falconer} \begin{equation}\label{dimF} \dim_{\rm H}F=\log_q(m).\end{equation} We associate to $F$ the $q$-automatic binary sequence $f$, which is the infinite fixed point $\varrho_{f}^\infty(1)$ of the substitution $\varrho_{f}$ defined by
\begin{equation*}
   \varrho_{f}: \, \begin{cases} 1 \mapsto 10^{a_2-(a_1+1)}1\cdots 0^{a_k-(a_{k-1}+1)}1\cdots 0^{a_m-(a_{m-1}+1)}1 0^{q-1-(a_{m})}\\ 
     0\mapsto 0^q\, , \end{cases} 
\end{equation*} where $0^j$ indicates a string of zeros of length $j$ with the convention that the string $0^0$ is equal to the empty string. The power series generating function $M_f(z):=\sum_{n\geqslant 0} f(n)\, z^n $ of $f$ has an infinite product representation, easily read off from the substitution; indeed, $M_f(z)= \prod_{j\geqslant 0}p_f(z^{q^j}),$ where $$p_f(z)=\sum_{k=1}^{m}z^{a_k}=1+z^{a_2}+z^{a_3}+\cdots+z^{a_{m-1}}+z^{a_m},$$ and where we have used the fact that $a_1=0$. This product representation shows that $M_f(z)$ is a Mahler function. 

In general, for an integer $q\geqslant 2$, a power series $M(z)\in\mathbb{Z}[[z]]$ is called a $q$-\emph{Mahler function} (or a just a \emph{Mahler function} when $q$ is understood) provided there is an integer $d\geqslant 1$ and polynomials $p_0(z),\ldots,p_d(z)\in\mathbb{Z}[z]$ with $p_0(z)p_d(z)\neq 0$ such that $M(z)$ satisfies the functional equation \begin{equation}\label{MFE} p_0(z)M(z)+p_1(z)M(z^q)+\cdots+p_d(z)M(z^{q^d})=0.\end{equation} We call the integer $q$, the {\em base} of the Mahler function; the minimal integer $d$ for which such an equation exists is called the {\em degree} of $M(z)$. So in the case of the sequence $f$ above, the function $M_f(z)$ is a degree-one $q$-Mahler function. The leading asymptotic behaviour of Mahler functions, as $z\to 1^-$, is governed by a number called the Mahler eigenvalue. This concept was introduced by Bell and Coons \cite{BC2017} in order to produce a quick transcendence test for Mahler functions. Our focus here is on its connection to the Hausdorff dimension of a related fractal, the self-affine set $F$.

To formalise this notion, suppose that $M(z)$ satisfies \eqref{MFE}, set $p_i:=p_i(1)$, and form the characteristic polynomial  of $M(z)$, $$\chi_M(\lambda):=p_0\lambda^d+p_1\lambda^{d-1}+\cdots+p_{d-1}\lambda+p_d.$$ Bell and Coons \cite{BC2017} showed that if $\chi_M(\lambda)$ has $d$ distinct roots, then there exists an eigenvalue $\lambda_f$ with $\chi_M(\lambda_M)=0$, which is naturally associated to $M(z)$. We call $\lambda_M$ the {\em Mahler eigenvalue} of $M(z)$. This natural association is best seen through the asymptotics. For example, in the case of degree-one $q$-Mahler functions, if $M(z)=\prod_{j\geqslant 0}p(z^{q^j})$ for some polynomial $p(z)\in\mathbb{Z}[z]$ with $p(0)=1$ and $p(1)\neq 0$, we have that $\lambda_M=p(1)$ and, as $z\to 1^-$, \begin{equation}\label{fos}M(z)\asymp\frac{1}{(1-z)^{\log_q p(1)}}(1+o(1)).\end{equation}

With the above concepts and definitions in hand, it is immediate that $$\lambda_f:=\lambda_{M_f}=p_f(1)=m,$$ since there are $m$ terms in the polynomial $p_f(z)$ each with coefficient equal to one. The validity of Theorem \ref{main1} is now readily apparent by combining this equality with \eqref{dimF}.

\begin{remark} Note that \eqref{fos} implies that as $z\to 1^-$ the function $(1-z)^{\log_q p(1)}M(z)$ potentially oscillates between two positive real numbers. In fact, this is the case---see, e.g., Brent, Coons and Zudilin \cite{BCZ2016}, Bell and Coons \cite{BC2017} and Coons \cite{C2017}---though little is known about the consequences and properties of this oscillation. In the light of Theorem \ref{main1}, it seems a reasonable question to ask if this oscillation has a relation to certain fractals or their properties. 
\end{remark}

\section{Fractal mass distributions are ghost measures}

In this section, we start by describing both the standard mass distribution $\nu_F$ supported on the attractor $F$ of an iterated function system $\mathcal{S}$ of rational affine contractions and the ghost measure $\mu_f$ of the associated automatic sequence $f$. We then prove Theorem \ref{main2}, that $\nu_F=\mu_f$, by showing that the Fourier--Stieltjes coefficients $\widehat{\nu_F}(n)$ and $\widehat{\mu_f}(n)$ are equal for every integer $n$. Details and results showing this equivalence can be found in Rudin \cite[Thm.~1.3.6]{Rudin}; see also Baake and Grimm \cite[Ch.~8]{AO} for a more recent discussion on Fourier analysis and measures focussing on those important to the study of aperiodic order.

Let $\mathcal{S}=\{S_1,\ldots,S_m\}$ be an iterated function system satisfying \eqref{S} and define, for any set $E\subseteq [0,1]$, the function $$S(E):=\bigcup_{i=1}^m S_i(E).$$ In particular, we have that $F:=\lim_{k\to\infty}S^k([0,1])$, where, as before, $S^k$ indicates $k$-fold function composition of $S$ with itself. Analogous to the construction of the Cantor set in Figure \ref{fig:cant}, we let $E_k:=S^k([0,1])$ be the level sets in the construction of $F$. Each of the level sets $E_k$ is a union of $m^k$ intervals of length $q^{-k}$. We divide the unit mass of the interval equally among each of these intervals and construct a mass distribution $\nu_F$ supported on the attractor $F$ as the limit of the mass distributions $\nu_{F,k}$ supported on the level sets $E_k$. Initially, we have that $\nu_{F,0}=\lambda\, \big|_{[0,1]}$ is Lebesgue measure restricted to the interval $[0,1]$, and for $k\geqslant 1$, we have \begin{equation}\label{nuk}\nu_{F,k}=\frac{q}{m}\cdot\nu_{F,k-1}\, \big|_{E_k}.\end{equation} Each measure $\nu_{F,k}$ is absolutely continuous with respect to Lebesgue measure. The measure $\nu_F$ is the (weak) limit of the $\nu_{F,k}$ as $k\to\infty$. 
For example, Figure \ref{fig:devil} contains the graph of the distribution function of $\nu_C$ over the interval $[0,1]$, known as the devil's staircase---there are uncountably many steps, one at each element of the Cantor set $C$.
\begin{figure}[htp]
\begin{center}
\includegraphics[width=4.7in,height=2.0in]{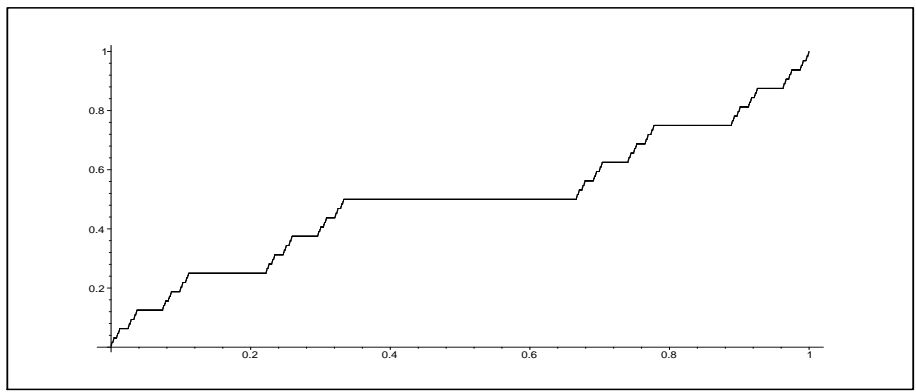}
\end{center}
\caption{The devil's staircase: the distribution function of $\nu_C=\mu_c$.}
\label{fig:devil}
\end{figure}

Turning now to the ghost measure of the related sequence $f$, for each $k$, analogous to the construction of $\mu_c$ in the Introduction, we take the sequence $f$ up to $q^k-1$ and reinterpret its (renormalised) values as the weights of a pure point probability measure $\mu_{f,k}$ on the torus $\mathbb{T}=[0,1)$ with support ${\rm supp}(\mu_{f,k})=\{j/q^k:0\leqslant j<q^k,\, f(j)=1\}.$ That is, $$\mu_{f,k}:= m^{-k}\sum_{j=0}^{q^k-1}f(j)\, \delta_{j/q^k},$$ where $\delta_x$ denotes the unit Dirac measure at $x$. The measures $\mu_{f,k}$ are the mass distributions associated to the finite sequences $\varrho_f^k(1)$, analogous to the mass distributions $\nu_{F,k}$ associated to the level sets $E_k$. The ghost measure $\mu_f$ is the (weak) limit of the sequence $(\mu_{f,k})_{k\geqslant 0}$ as $k\to\infty$.

With this terminology in hand, we are now ready to prove Theorem \ref{main2}.

\begin{proof}[Proof of Theorem \ref{main2}] As stated at the beginning of this section, to show that $\nu_F=\mu_f$, it is enough (equivalent, actually) to show that their Fourier--Stieltjes coefficients are equal; that is, for every integer $n$, $\widehat{\nu_F}(n)=\widehat{\mu_f}(n)$. 

For $\nu_F$, we note that given an integer $n$, applying \eqref{nuk} followed by a change of variable, we have \begin{multline*} \widehat{\nu_{F,k}}(n) := \int_0^1 \ee^{2\pi \ii nx}\dd \nu_{F,k}(x)=\frac{q}{m}\int_{0}^1\ee^{2\pi \ii nx}\dd \nu_{F,k-1}(x)\,\big|_{E_k}\\
=\, \frac{1}{m}\sum_{j=1}^m \int_0^1 \ee^{2\pi \ii n\left(\frac{a_j+y}{q}\right)}\dd \nu_{F,k-1}(y)
=\, \frac{1}{m}\sum_{j=1}^m \ee^{2\pi \ii n a_j/q} \int_0^1 \ee^{2\pi \ii (n/q) y}\dd \nu_{F,k-1}(y)\\=\left(\frac{1}{m}\cdot p_f\left(\ee^{2\pi \ii n/q}\right)\right)\, \widehat{\nu_{F,k-1}}(n/q)
= \prod_{\ell=1}^k \left(\frac{1}{m}\cdot p_f\left(\ee^{2\pi \ii n/q^\ell}\right)\right)\cdot\int_0^1 \ee^{2\pi \ii (n/q^k)x}\dd x,
\end{multline*} where the last equality follows by repeated application of the penultimate equality, and the fact that $\nu_{F,0}(x)=\lambda\,\big|_{[0,1]}$. Taking the limit as $k\to\infty$, we thus have \begin{equation}\label{hatnu}\widehat{\nu_{F}}(n)=\prod_{\ell\geqslant 1} \left(\frac{1}{m}\cdot p_f\left(\ee^{2\pi \ii n/q^\ell}\right)\right),\end{equation} since $\lim_{k\to\infty}\int_0^1 \ee^{2\pi \ii (n/q^k)x}\dd x=1$ for any integer $n$.

The calculation for $\mu_f$ is more straight-forward. Indeed, by inspecting the iterates $\varrho_f^k(1)$, it is clear that $$\sum_{j=0}^{q^k-1}f(j)\, z^j=\prod_{\ell=0}^{k-1} p_f\left(z^{q^\ell}\right),$$
so that 
\begin{equation*}
\widehat{\mu_{f,k}}(n):= m^{-k}\sum_{j=0}^{q^k-1}f(j)\, \ee^{-2\pi\ii n j/q^k}=m^{-k}\prod_{\ell=0}^{k-1} p_f\left(\ee^{-2\pi\ii n q^{\ell-k}}\right)=\prod_{\ell=1}^{k}\left(\frac{1}{m}\cdot p_f\left(\ee^{-2\pi\ii n /q^{\ell}}\right)\right).
\end{equation*} Taking the limit as $k\to\infty$, $\widehat{\mu_f}(n)$ is equal to the righthand side of \eqref{hatnu} and so also to $\widehat{\nu_F}(n)$, thus $\mu_f=\nu_F$, which finishes the proof of the theorem.
\end{proof}

\section{Further comments}

The idea of fractals arising from automatic sequences is by no means new. The limiting sets of automatic sequences and their relation to fractals was investigated by Barb\'e and von Haeseler \cite{BvH2003}. In fact, even our paradigmatic example---the Cantor sequence---has been around for some time. It played an emphatic role in the paper of Allouche and Skordev \cite{AS2006}; one of the purposes of their paper was to remind the mathematical community that there is an abundance of literature connecting automatic sequences to fractals. See their paper \cite{AS2006} for a detailed reference list. There are even several examples given in Allouche and Shallit~\cite{ASbook}, e.g., the dragon curve is related to the regular paperfolding sequence \cite[p.~155]{ASbook} and one can obtain von Koch's famous snowflake via an $8$-automatic sequence on three letters \cite[p.~202]{ASbook}. As well, the notion of relating the Hausdorff dimension to properties of automatic fractals is not new; Adamczewski and Bell explored the possibility of a relationship between the Hausdorff dimension and entropy of an automatic fractal, which led them to ask several questions; see \cite{AB2011}.

In this paper, we demonstrated how two objects, attractors of certain rational affine iterated function systems and automatic sequences, can be related and explored a few of the consequences of that relationship. As described above, the connection between fractal and automatic sequences is well-established. Our contribution is an explicit relationship of some invariants related to either object. In particular, we showed that the Hausdorff dimension of the attractor is related to the asymptotics of the generating function of the automatic sequence. The class of iterated function systems we considered has attractors that are self-affine sets, and so the relationship we describe is with a very tangible, and well-understood, geometric object. That said, there is no reason not to desire a generalisation beyond this straight-forward class of examples. Ghost measures exist for a much larger class of sequences, e.g., for a large subset of regular sequences \cite{CMpre} as defined by Allouche and Shallit \cite{AS1992}. The ghost measures of the so-called `affine' $2$-regular sequences have been recently completely described by Evans \cite{EvansPRE}. The existence of such measures begs the question of further relationships between automatic and regular sequences to fractals. 

\bibliographystyle{plain}



\end{document}